\documentclass[12pt]{amsart}
\usepackage{amssymb}
\usepackage[mathscr]{eucal}
\usepackage{epsf}
\usepackage{epsfig}
\usepackage{color}
\DeclareGraphicsExtensions{.pstex,.eps,.epsi,.ps}

 \newtheorem{thm}{Theorem}[section]
 
 \newtheorem{lem}[thm]{Lemma}
 \newtheorem{prop}[thm]{Proposition}
 \theoremstyle{definition}
 
 \theoremstyle{remark}
 
 \numberwithin{equation}{section}

 \newcommand{\J}{\mathcal{J}}

 \newcommand{\Real}{\mathbb{R}}

 \newcommand{\tr}{\textbf{tr}}
 \newcommand{\di}{\textbf{div}}
 \newcommand{\vol}{\textbf{vol}}
 \newcommand{\ric}{\textbf{Rc}}

 \newcommand{\Rm}{\textbf{Rm}}

\newcommand{\LD}{\mathcal L}

\newcommand{\cJ}{\mathcal J}
\newcommand{\TW}{\overline\nabla}
\newcommand{\twR}{\overline{\Rm}}

\begin{document}

\title[Volume Rigidity of $S^1\to \mathcal P \to\mathbb{CP}^n$]{Volume Rigidity of Principal Circle Bundles over the Complex Projective Space}

\author{Paul W.Y. Lee}
\email{wylee@math.cuhk.edu.hk}
\address{Room 216, Lady Shaw Building, The Chinese University of Hong Kong, Shatin, Hong Kong}

\date{\today}

\begin{abstract}
In this paper, we prove that principal circle bundles over the complex projective space equipped with the standard Sasakian structures are volume rigid among all $K$-contact manifolds satisfying positivity conditions of tensors involing the Tanaka-Webster curvature.
\end{abstract}

\maketitle


\section{Introduction}

One of the main consequences of the classical Bishop-Gromov inequality states that a Riemannian manifold with the Ricci curvature bounded below by a positive constant has its volume controlled by the corresponding model which is the sphere of constant curvature. Moreover, equality of the volumes holds only if the manifold is isometric to the model. This result is also the starting point of the theory of almost rigidity (see \cite{ChCo} and reference therein).

Recently, there is a surge of interest in extending various comparison type results to the sub-Riemannian setting (see \cite{AgLe1,AgLe2,BaGa,BaBoGa,LeLiZe,Le1,Le2,BaGrKuTh}). However, extension of the above rigidity result to the sub-Riemannian setting seems to be missing. One of the purposes of this paper is to provide such a result for $K$-contact manifolds.

Recall that a contact metric manifold is a contact manifold $M$ equipped with a Riemannian metric $\left<\cdot,\cdot\right>$ and $(1,1)$-tensor $\cJ$ (which is a complex structure defined on the contact distribution) satisfying some compatibility conditions (see Section 2 for more details). It is $K$-contact if the Reeb vector field $\xi$ is Killing.

An example of a $K$-contact metric manifold is given by the Hopf fibration which is a principal bundle $S^1\to S^{2n+1}\to\mathbb{CP}^n$. There is a $K$-contact structure which make $S^{2n+1}$ a Sasakian manifold. Moreover, the quotient $\mathbb{CP}^n$ inherits from this Sasakian structure a K\"ahler structure which a constant multiple of the standard one given by the Fubini-Study metric. We denote the total space and the base of the Hopf fibration equipped with this structure by $\mathcal P(k,1)$ and $\mathcal B(k,1)$, respectively. Here $k$ is the number given by $\left<\Rm(\cJ X, X)X,\cJ X\right>=k^2$, where $\Rm$ is the Riemann curvature tensor of $\mathcal B(k,1)$ and $X$ is any unit tangent vector. Finally, let $\mathcal P(k,m)=\mathcal P(k,1)/\mathbb{Z}_m$, where $\mathbb{Z}_m$ acts on $\mathcal P(k,1)$ by a discrete subgroup of $S^1$ of order $m$. These are all the model spaces of our volume rigidity results. (See Section \ref{Hopf} for more details.)

In order to state the main result, we also need the Tanaka-Webster curvature $\overline{\Rm}$ which is the curvature of a connection defined in the CR case in \cite{Tana} and more general contact case in \cite{Tann}. The tensor $\mathcal N$ is defined by $\mathcal N(X)=P(\nabla_X\cJ)X$, where $P$ is the orthogonal projection onto the contact distribution and $\nabla$ is the Levi-Civita connection.

\begin{thm}\label{main}
Assume that the above assumptions hold and that the following curvature conditions are satisfied:
\begin{enumerate}
\item $\left<\overline{\Rm}(\cJ X, X)X,\cJ X\right>-|\mathcal N(X)|^2\geq k^2$ for all unit tangent vectors $X$,
\item $\tr_{\{X,\cJ X\}^\perp}(Y\mapsto\left<\overline{\Rm}(Y, X)X, Y\right>-|\mathcal N(X)|^2\geq \frac{(2n-2)k^2}{4}$ for all unit tangent vectors $X$.
\end{enumerate}
Then
\[
m:=\frac{\text{len}(\bar C)}{\text{len}(C)}\leq \frac{\vol(\mathcal P(k,1))}{\vol(M)},
\]
where $C$ and $\bar C$ are closed orbits of the Reeb fields on $M$ and $\mathcal P(k,1)$, respectively. Moreover, if equality holds, then $m$ is a positive integer and $M$ is isometrically contactomorphic to $\mathcal P(k,m)$.
\end{thm}

Here isometrically contactomorphic means there is an isometry between the two spaces which is also a contactomorphism. Remark that the curvature conditions appeared earlier in a related work of the author \cite{LeLiZe}.

We also remark that the existence of a closed orbit in a compact contact manifold is a long standing open problem called the Weinstein conjecture \cite{We} (see \cite[Chapter 2]{Bl} for a brief discussion and numerous references). In the case of $K$-contact manifolds, there are in fact multiple closed orbits (see \cite{Ru} and references therein).

A symplectic analogue of the above result follows immediately by considering the Boothby-Wang fibrations \cite{BoWa}. First, we recall that a contact manifold is regular if the flow of the corresponding Reeb vector field $\xi$ is regular. On a compact manifold, it means that each orbit of $\xi$ is closed. Assume that the contact manifold is regular. A result in \cite{BoWa} shows that one can multiply a positive function $f$ to the Reeb field $\xi$ such that $M$ is the total space of a principal circle bundle $\pi:M\to N$, called a Boothby-Wang fibration, with action defined by the flow of $f\xi$. Moreover, the base space $N$ of this bundle is an integral symplectic manifold. If $\omega$ is the symplectic form on $N$, then $\pi^*\omega$ is the exterior derivative of the new contact form $\eta$. Conversely, a result \cite{Ko} shows that any integral symplectic manifold is the base of a Boothby-Wang fibration. If $\cJ_N$ defines an almost complex structure with a compatible Riemannian metric $\left<\cdot,\cdot\right>^N$, then the lifts of these structures to $M$ together with $\eta$ define a $K$-contact manifold (see \cite{Bl} and references therein for further details).

Let $M$ be an integral symplectic manifold equipped with an almost complex structure $\cJ$ and a compatible Riemannian metric $\left<\cdot,\cdot\right>$. The following result is a consequence of the above discussion and Theorem \ref{main} (the same notations are used for both the symplectic and the contact case though it will be clear from the context which case we are in).

\begin{thm}\label{mains}
Assume that the following curvature conditions hold:
\begin{enumerate}
\item $\left<\Rm(\cJ X, X)X,\cJ X\right>-|\mathcal N(X)|^2\geq k^2$ for all unit tangent vectors $X$,
\item $\tr_{\{X,\cJ X\}^\perp}(Y\mapsto\left<\Rm(Y, X)X, Y\right>-|\mathcal N(X)|^2\geq \frac{(2n-2)k^2}{4}$ for all unit tangent vectors $X$,
\end{enumerate}
where $\mathcal N$ is the tensor defined by $\mathcal N(X)=(\nabla_X\cJ)X$.

Then the volume of $M$ is bounded above by that of $\mathcal B(k)$. Moreover, equality holds only if $M$ is isometrically symplectic to $\mathcal B(k)$.
\end{thm}

Note that comparison results for K\"ahler manifolds are well-studied (see \cite{LiWa} and references therein) but not in the symplectic case.

The organization of the paper is as follows. In section \ref{subintro}, we recall various notions on contact metric manifolds and the corresponding sub-Riemannian geodesic flows needed in this paper. In section \ref{Hopf}, we give a brief discussion on the model spaces, circle bundles over $\mathbb{CP}^n$. Section \ref{SMyers} is devoted to the proof of some Myers' type maximal diameter theorems. In Section \ref{SClosedOrbit}, we prove a few comparison type results for a closed Reeb orbit in the spirit of \cite{HeKa}. The equality case of these estimates is discussed in Secton \ref{SEqual}. Finally, we finish the proof of Theorem \ref{main} in Section \ref{endproof}. The appendix summarizes several formulas relating the structures defined on the contact metric manifolds which are needed in this paper.

\smallskip

\section{Contact Manifolds and their Sub-Riemannian Geodesics}
\label{subintro}
In this section, we recall several notions about contact manifolds and the sub-Riemannian geodesics which are needed for this paper (see \cite{Bl} and references therein for more details about Riemannian geometry of contact manifolds and see \cite{Mo} for sub-Riemannian geometry). Let $M$ be a $2n+1$ dimensional manifold equipped with a contact form $\eta$ (i.e. $d\eta$ is non-degenerate on $\ker\eta$). Let $\xi$ be the corresponding Reeb field $\xi$ defined by the conditions $\eta(\xi)=1$ and $d\eta(\xi,\cdot)=0$. A smoothly varying inner product defined on $\ker\eta$ is a called a Carnot-Caratheodory or sub-Riemannian metric $\left<\cdot, \cdot\right>$ on $\ker\eta$. This can be extended to a Riemannain metric, denoted by the same symbol $\left<\cdot, \cdot\right>$, by the conditions $\left<\xi, X\right>=0$ and $|\xi|=1$ for all $X$ in $\ker\eta$. A $(1,1)$-tensor $\cJ$ together with $\xi$, $\eta$, and the Riemannian metric $\left<\cdot,\cdot\right>$ is an contact metric structure if $\cJ\xi=0$, $\cJ^2(X_1)=-X_1$, $\left<\cJ X_1,\cJ X_2\right>$, and $d\eta(X_1,X_2)=\left<X_1,JX_2\right>$ for all $X_1$ and $X_2$ in $\ker\eta$. Finally, let $h$ be the $(1,1)$-tensor defined by $h=\LD_\xi \cJ$. A contact metric manifold is $K$-contact if $\xi$ is an isometry and this is equivalent to $h=0$. It is a CR manifold if
\[
\nabla_u \cJ(v)=\frac{1}{2}\left<u+hu,v\right>\xi-\frac{1}{2}\left<\xi,v\right>(u+hu).
\]
A $K$-contact CR manifold is Sasakian.

By the Chow-Rashevskii Theorem \cite{Chow,Ra}, any two points can be connected by a horizontal curve (i.e. a curve tangent to $\ker\eta$). The length of the shortest horizontal curve connecting two points $x_0$ and $x_1$ in $M$ is called the Carnot-Caratheodory or sub-Riemannian distance between $x_0$ and $x_1$. It is denoted by $d_{CC}(x_0,x_1)$. The function $g(\cdot)=d_{CC}(x_0,\cdot)$ is locally semi-concave outside the diagonal \cite{CaRi} and so it is twice differentiable Lebesgue almost everywhere \cite{EvGa,CaSi}. Moreover, since there is no abnormal minimizer (see \cite{Mo} for more detail), the function $g$ is $C^\infty$ along a sub-Riemannian geodesic except at the end-points.

The function $g$ satisfies the equation
\begin{equation}\label{eikonal}
|\nabla_H g(x)|=1
\end{equation}
for each $x$ where $g$ is differentiable. The vector field $\nabla_Hg$ is the horizontal gradient of $g$ which is defined as the orthogonal projection of the gradient vector field $\nabla g$ onto the distribution $\ker\eta$. Here the gradient is taken with respect to the Riemannian metric defined above. Moreover, if $\gamma$ is a geodesic which starts from $x$ and ends at $x_0$, then \begin{equation}\label{geograd}
\dot\gamma(t)=-\nabla_Hg(\gamma(t)).
\end{equation}
A computation using this relation gives

\begin{lem}\label{geo}
A sub-Riemannian geodesic $\gamma$ satisfies the following system of equations:
\[
\begin{split}
&\frac{D^2}{dt^2}\gamma(t)=a(t) \cJ\dot\gamma(t) -\frac{h(\dot\gamma(t), \cJ\dot\gamma(t))}{2}\xi(\gamma(t)),
\end{split}
\]
\[
\dot a(t)=\frac{1}{2}\left<h(\dot\gamma(t)), \cJ\dot\gamma(t)\right>,
\]
where $\frac{D}{dt}$ denotes the covariant derivative of the Riemannian metric defined above.
\end{lem}

\begin{proof}
By differentiating (\ref{eikonal}) and applying Proposition \ref{h},
\[
\begin{split}
&0=\left<\nabla^2g(\nabla g),X\right>-\left<\nabla g,\xi\right>\left<\nabla^2g(\xi),X\right>-\left<\nabla g,\xi\right>\left<\nabla g,\nabla_X\xi\right>\\
&=\left<\nabla^2g(\nabla_H g),X\right>-\frac{\left<\nabla g,\xi\right>}{2}\left<\cJ\nabla g+h\cJ\nabla g,X\right>.
\end{split}
\]

By differentiating (\ref{geograd}) and applying Proposition \ref{h},
\[
\begin{split}
&\frac{D^2}{dt^2}\gamma=-\nabla^2g(\dot\gamma(t)) +\left<\nabla^2g(\dot\gamma(t)),\xi\right>_{\gamma(t)}\xi(\gamma(t))\\
& +\left<\nabla g,\nabla_{\dot\gamma(t)}\xi\right>_{\gamma(t)}\xi(\gamma(t)) +\left<\nabla g,\xi\right>_{\gamma(t)}\nabla_{\dot\gamma(t)}\xi(\gamma(t))\\
&=-\frac{\left<\nabla g,\xi\right>_{\gamma(t)}}{2}(\cJ +h\cJ)\dot\gamma(t) \\
&+\frac{1}{2}\left<\nabla g,\cJ h\nabla g\right>_{\gamma(t)}\xi(\gamma(t)) -\frac{1}{2}\left<\nabla g,\xi\right>_{\gamma(t)}(\cJ+\cJ h)\dot\gamma(t)\\
&=-\left<\nabla g,\xi\right>_{\gamma(t)}\cJ\dot \gamma(t) -\frac{1}{2}\left<\nabla g,h\cJ \nabla g\right>_{\gamma(t)}\xi(\gamma(t))
\end{split}
\]
The first assertion follows with $a(t)=-\left<\nabla g,\xi\right>_{\gamma(t)}$. The second one follows from
\[
\begin{split}
&\frac{d}{dt}\left<\nabla g,\xi\right>_{\gamma(t)} =\left<\nabla^2 g(\dot\gamma(t)),\xi\right>_{\gamma(t)}+\left<\nabla g,\nabla_{\dot\gamma(t)}\xi\right>_{\gamma(t)} =\frac{1}{2}\left<\dot\gamma(t),\cJ h \dot\gamma(t)\right>.
\end{split}
\]
\end{proof}

Next, we define a family of orthonormal frames along a sub-Riemannian geodesic.

\begin{lem}
There is a family of orthonormal frames
\[
v(t)=(v_0(t),v_1(t), v_2(t), ..., v_{2n}(t))^T
\]
defined along the geodesic $t\mapsto\gamma(t)$ which span the orthogonal complements of $\dot\gamma(t)$ such that $v_0(t)=\xi(\gamma(t))$, $v_1(t)=\dot\gamma(t)$, $v_2(t)=\cJ\dot\gamma(t)$, and
\[
\dot v(t)=W(t)v(t),
\]
where
\begin{enumerate}
\item $a(t)=-\left<\nabla g,\xi\right>_{\gamma(t)}$,
\item $H_{ij}(t)=\left<h(v_i(t)), v_j(t)\right>$,
\item $N_i(t)=\left<(\nabla_{v_1(t)}\cJ)v_1(t), v_i(t)\right>$,
\item $W(t)=\left(\begin{array}{cc}
W_1(t) & W_2(t)U(t)^T\\
-U(t)W_2(t)^T & O
\end{array}\right)$,
\item $U(t)$ is a family of $(2n-2)\times(2n-2)$ orthogonal matrices,
\item $W_1(t)=\left(\begin{array}{ccc}
0 & \frac{H_{12}(t)}{2} & -\frac{1+H_{11}(t)}{2}\\
-\frac{H_{12}(t)}{2} & 0 & a(t)\\
\frac{1+H_{11}(t)}{2} & -a(t) & 0
\end{array}\right)$,
\item $W_2(t)=\left(\begin{array}{ccc}
\frac{H_{32}(t)}{2} & ... & \frac{H_{2n, 2}(t)}{2}\\
0 & ... & 0\\
N_3(t) & ... & N_{2n}(t)
\end{array}\right)$.
\end{enumerate}
\end{lem}

\begin{proof}
A computation shows that
\[
\begin{split}
&\dot v_0(t)=\nabla_{\dot\gamma(t)}\xi=-\frac{1}{2}(\cJ+\cJ h)\dot\gamma(t)=-\frac{1}{2}(v_2(t)-hv_2(t)),\\
&\dot v_1(t)=a(t)v_2(t) -\frac{H_{12}(t)}{2}v_0(t),\\
&\dot v_2(t)=(\nabla_{\dot\gamma(t)} \cJ)v_1(t)+\cJ\dot v_1(t)=(\nabla_{v_1(t)} \cJ)v_1(t) -a(t)v_1(t).
\end{split}
\]
Note that $N_0(t)=\frac{1}{2}(1+H_{11}(t))$ and $N_1=N_2=0$ by Proposition \ref{Rm}.

Let $\bar v_3(t),...,\bar v_{2n}(t)$ be a family of bases for the orthogonal complement of $\{v_0(t),v_1(t),v_2(t)\}^\perp$. Let $A(t)$ be the family of matrices defined by $A_{ij}(t)=\left<\dot{\bar v}_i(t), {\bar v}_j(t)\right>$, where $i, j=3,...,2n+1$. Let $U(t)$ be the family of orthogonal matrices defined by $U(0)=I$ and $\dot U(t)=-U(t)A(t)$. Finally, let $v_i(t)=\sum_{j=3}^{2n+1}U_{ij}(t)\bar v_j(t)$.
\[
\begin{split}
\left(\begin{array}{c}
\dot v_3(t)\\
\vdots\\
\dot v_{2n+1}(t)
\end{array}\right)&= (\dot U(t)+U(t)A(t))\left(\begin{array}{c}
\bar v_3(t)\\
\vdots\\
\bar v_{2n+1}(t)
\end{array}\right)-U(t)W_2(t)^T\left(\begin{array}{c}
v_0(t)\\
v_1(t)\\
v_{2}(t)
\end{array}\right)\\
&=-U(t)W_2(t)^T\left(\begin{array}{c}
v_0(t)\\
v_1(t)\\
v_{2}(t)
\end{array}\right)
\end{split}
\]
The assertion follows.
\end{proof}

Using the above frames, we can show that the Hessian of $g$ satisfies a matrix Riccati equation.

\begin{lem}\label{S1}
Let
\begin{enumerate}
\item $S_{ij}(t)=-\left<\nabla^2 g(v_i(t)), v_j(t)\right>$,
\item $H_{ij}(t)=\left<h(v_i(t)), v_j(t)\right>$,
\item $J_{ij}(t)=\left<\J v_i(t), v_j(t)\right>$,
\item $S_1(t)=S(t)+\frac{a(t)}{2}H(t)J(t)$,
\item $R_{ij}(t)=\left<\Rm(v_i(t), \dot\gamma(t))\dot\gamma(t), v_j(t)\right>$,
\item $K_{1, ij}(t)=\left<(\nabla_{v_j(t)}\J)v_i(t)+(\nabla_{v_i(t)}\J)v_j(t), v_1(t)\right>$,
\item $C_i=\left(\begin{array}{cc}
\bar C_i & 0\\
0 & 0_{2n-2}
\end{array}\right)$ with $i=1,2$,
\item $\bar C_1=\left(\begin{array}{ccc}
0 & 0 & 0\\
0 & 0 & 0\\
0 & 0 & 1
\end{array}\right)$,
\item $\bar C_2=\left(\begin{array}{ccc}
0 & 0 & 0\\
0 & 0 & 0\\
1 & 0 & 0
\end{array}\right)$,
\item $C_3=\left(\begin{array}{cc}
\bar C_3 & 0\\
0 & I_{2n-2}
\end{array}\right)$,
\item $\bar C_3=\left(\begin{array}{ccc}
0 & 0 & 0\\
0 & 1 & 0\\
0 & 0 & 1
\end{array}\right)$.
\end{enumerate}
Then, for all $t$ in the open interval $(0, 1)$,
\[
\begin{split}
&\dot S_1(t) -\left(W(t)-\frac{a(t)}{2}J(t) +\frac{1}{2}(I +H(t))C_2\right)S_1(t)\\
& -S_1(t)\left(W(t)-\frac{a(t)}{2}J(t)  +\frac{1}{2}(I +H(t))C_2\right)^T +S_1(t)C_3S_1(t)\\
&=-R(t) +\frac{H_{12}(t)}{4} H(t)J(t)+\frac{1}{4}(I+H(t))C_1(I+H(t)) -\frac{a(t)^2}{4}C_3 -\frac{a(t)}{2}K_1(t).
\end{split}
\]
\end{lem}

\begin{proof}
By differentiating (\ref{eikonal}) twice,
\[
\begin{split}
&0=\nabla^3g(\nabla_H g,Y,X)+\left<\Rm(Y,\nabla_H g)\nabla g,X\right>+\nabla^2g(\nabla^2 g(Y),X)\\
&-\nabla^2 g(\xi,Y)\nabla^2g(\xi,X) -\left<\nabla g,\nabla_Y\xi\right>\nabla^2g(\xi,X)\\
&-\left<\nabla g,\xi\right>\nabla^2g(\nabla_Y\xi,X)-\frac{\left<\nabla^2 g(Y),\xi\right>}{2}\left<\cJ\nabla g+h\cJ\nabla g,X\right>\\
&-\frac{\left<\nabla g,\nabla_Y\xi\right>}{2}\left<\cJ\nabla g+h\cJ\nabla g,X\right>\\
&-\frac{\left<\nabla g,\xi\right>}{2}\left<\cJ\nabla^2 g(Y)+h\cJ\nabla^2 g(Y),X\right>\\
&-\frac{\left<\nabla g,\xi\right>}{2}\left<\nabla_Y(\cJ+h\cJ)\nabla g,X\right>
\end{split}
\]
By setting $Y=v_i(t)$, $X=v_j(t)$, and $S_{ij}(t)=-\nabla^2g(v_i(t),v_j(t))$,
\[
\begin{split}
&0=-\nabla^3g(v_1(t),v_i(t),v_j(t))+R_{ij}(t)\\
&+a(t)\left<\Rm(v_i(t),v_1(t))v_0(t),v_j(t)\right>+\sum_{k\neq 0}S_{ik}(t)S_{kj}(t)\\
&-\frac{1}{2}(\delta_{i2}+H_{i2}(t))S_{0j}(t) +\frac{a(t)}{2}\sum_k\left(J_{ik}(t)+\sum_lH_{il}(t)J_{lk}(t)\right)S_{kj}(t) \\
&-\frac{S_{0i}(t)}{2}(\delta_{2j}+H_{2j}(t))-\frac{(\delta_{i2}+H_{i2}(t))(\delta_{j2}+H_{j2}(t))}{4}\\
&-\frac{a(t)}{2}\sum_kS_{ik}(t)\left(J_{kj}(t)+\sum_lJ_{kl}(t)H_{lj}(t)\right)\\
&+\frac{a(t)^2}{2}\left<(\cJ+h\cJ)\nabla_{v_i(t)}\xi,v_j(t)\right>\\ &-\frac{a(t)}{2}\left<\nabla_{v_i(t)}(\cJ+h\cJ)v_1(t),v_j(t)\right>
\end{split}
\]
On the other hand,
\[
\begin{split}
&\frac{d}{dt}\nabla^2g_{\gamma(t)}(v_i(t),v_j(t)) =\nabla^3g_{\gamma(t)}(\dot\gamma(t),v_i(t),v_j(t))\\ &+\sum_kW_{ik}(t)\nabla^2g_{\gamma(t)}(v_k(t),v_j(t)) +\sum_kW_{jk}(t)\nabla^2g_{\gamma(t)}(v_i(t),v_k(t)).
\end{split}
\]
Therefore,
\[
\begin{split}
&-\dot S_{ij}(t)+\sum_kW_{ik}(t)S_{kj}(t)+\sum_kW_{jk}(t)S_{ki}(t)\\
&=R_{ij}(t)+a(t)\left<\Rm(v_i(t),v_1(t))v_0(t),v_j(t)\right>+\sum_{k\neq 0}S_{ik}(t)S_{kj}(t)\\
&-\frac{1}{2}(\delta_{i2}+H_{i2}(t))S_{0j}(t) +\frac{a(t)}{2}\sum_k\left(J_{ik}(t)+\sum_lH_{il}(t)J_{lk}(t)\right)S_{kj}(t) \\
&-\frac{S_{0i}(t)}{2}(\delta_{2j}+H_{2j}(t))-\frac{(\delta_{i2}+H_{i2}(t))(\delta_{j2}+H_{j2}(t))}{4}\\
&-\frac{a(t)}{2}\sum_kS_{ik}(t)\left(J_{kj}(t)+\sum_lJ_{kl}(t)H_{lj}(t)\right)\\
&+\frac{a(t)^2}{2}\left<(\cJ+h\cJ)\nabla_{v_i(t)}\xi,v_j(t)\right>\\ &-\frac{a(t)}{2}\left<\nabla_{v_i(t)}(\cJ+h\cJ)v_1(t),v_j(t)\right>
\end{split}
\]

By Proposition \ref{J} and \ref{Rm},
\[
\begin{split}
&-\frac{1}{2}\left<v_1(t), \nabla_{v_i(t)}(\cJ+\cJ h)v_j(t)\right> -\left<\Rm(v_i(t), v_1(t))v_0(t), v_j(t)\right>\\
&=-\frac{1}{2}\left<(\nabla_{v_j(t)}\cJ)(v_i(t))+(\nabla_{v_i(t)}\cJ)v_j(t), v_1(t)\right>-\frac{1}{2}\left<\nabla_{v_1(t)}(\cJ h)v_i(t), v_j(t)\right>.
\end{split}
\]
It follows that
\[
\begin{split}
&-\dot S_{ij}(t)+\sum_kW_{ik}(t)S_{kj}(t)+\sum_kW_{jk}(t)S_{ki}(t)\\
&=R_{ij}(t) +\frac{a(t)}{2}\left<\nabla_{v_i(t)}\cJ(v_j(t)),v_1(t)\right> +\frac{a(t)}{2}\left<\nabla_{v_j(t)}\cJ(v_i(t)),v_1(t)\right>\\ &+\frac{a(t)}{2}\left<\nabla_{v_1(t)}(\cJ h)(v_i(t)),v_j(t)\right> +\sum_{k\neq 0}S_{ik}(t)S_{kj}(t)\\
&-\frac{1}{2}(\delta_{i2}+H_{i2}(t))S_{0j}(t) +\frac{a(t)}{2}\sum_k\left(J_{ik}(t)+\sum_lH_{il}(t)J_{lk}(t)\right)S_{kj}(t) \\
&-\frac{S_{0i}(t)}{2}(\delta_{2j}+H_{2j}(t)) -\frac{(\delta_{i2}+H_{i2}(t))(\delta_{j2}+H_{j2}(t))}{4}\\
&-\frac{a(t)}{2}\sum_kS_{ik}(t)\left(J_{kj}(t)+\sum_lJ_{kl}(t)H_{lj}(t)\right)\\
&-\frac{a(t)^2}{4}\left<(\cJ+h\cJ)(\cJ+\cJ h)v_i(t),v_j(t)\right>.
\end{split}
\]

In other words,
\[
\begin{split}
&-\dot S(t)+W(t)S(t)+S(t)W(t)^T\\
&=R(t) +\frac{a(t)}{2}K_1(t)\\
&+\frac{a(t)}{2}\left(\frac{d}{dt}(H(t)J(t))-W(t)H(t)J(t)-H(t)J(t)W(t)^T\right)\\
& +S(t)C_3S(t)-\frac{1}{2}(I+H(t))C_2S(t) +\frac{a(t)}{2}(J(t)+H(t)J(t))S(t) \\
&-\frac{1}{2}S(t)C_2^T(I+H(t))-\frac{1}{4}(I+H(t))C_1(I+H(t))\\
&-\frac{a(t)}{2}S(t)(J(t)+J(t)H(t))+\frac{a(t)^2}{4}(C_3+2H(t)+H(t)^2).
\end{split}
\]

By rewriting this in terms of $S_1(t)$ and using $C_2H(t)=0$,
\[
\begin{split}
&0=\dot S_1(t)+R(t) +\frac{a(t)}{2}K_1(t)-\frac{H_{12}(t)}{4}H(t)J(t) +S_1(t)C_3S_1(t)\\
&-S_1(t)\left(W(t)-\frac{a(t)}{2}J(t)+\frac{1}{2}(I+H(t))C_2\right)^T\\
&-\left(W(t)-\frac{a(t)}{2}J(t)+\frac{1}{2}(I+H(t))C_2\right)S_1(t)\\ &-\frac{1}{4}(I+H(t))C_1(I+H(t)) +\frac{a(t)^2}{4}C_3
\end{split}
\]
as claimed.
\end{proof}

\smallskip

\section{On Principal Circle Bundles Over $\mathbb{CP}^n$}\label{Hopf}

In this section, we will give a brief discussion on the model space, circle bundles over the complex projective space (see \cite{Bl,Ka,Mo} for further details).

The Hopf fibration is a principal circle bundle $S^1\to S^{2n+1}\to\mathbb{CP}^n$. We consider the total space $S^{2n+1}$ as the subset of the complex vector space $\mathbb{C}^{n+1}$
\[
S^{2n+1}=\left\{z\in\mathbb{C}^{n+1}||z|^2=4\right\}.
\]
The circle action is defined by $\theta\mapsto e^{-i\theta/2}z$. Its infinitesimal generator $-iz/2$ defines the Reeb vector field of the contact structure and its orbits have length $4\pi$ (assuming that the Reeb field has length 1).

The Riemannian metric $\left<\cdot,\cdot\right>$, the contact form $\eta$, and the tensor $\cJ$ are defined by $\left<v,w\right>=\text{Re}\left(\sum_{i=1}^{n+1}\bar v_iw_i\right)$, $\eta(w)=\left<-\frac{iz}{2},w\right>$, and $\cJ v=iv$, respectively, where $v=(v_1,...,v_{n+1})$.

Points in $S^{2n+1}$ are of the form
\[
(2\cos\theta \,z_0,2\sin\theta \,z')
\]
where $|z_0|=|z'|=1$. The points $(2\cos\theta z_0,0,...,0)$ and $(0,2z')$ are joined by the unit speed (Riemannian or sub-Riemannian) geodesic $t\mapsto (2\cos(t/2)z_0,2\sin(t/2)z')$ of length $\pi$.

The standard Euclidean structure on $\mathbb{C}^{n+1}$ induces a Riemannian metric on $S^{2n+1}$ with constant sectional curvature $\frac{1}{4}$. It induces a metric on the quotient $\mathbb{CP}^n$ such that the projection map is a Riemannian submersion. By the formula in \cite{On}, the curvature on $\mathbb{CP}^n$ is given by $\left<\Rm(Y,X)X,Y\right>=\frac{1}{4}+\frac{3}{4}\left<X,JY\right>^2$ for all unit tangent vectors $X$ and $Y$ such that $\left<X,Y\right>=0$. So this Riemannian structure on $\mathbb{CP}^n$ is a constant multiple of the Fubini-Study metric.

If we multiply this Riemannian metric on $\mathbb{CP}^n$ and the contact form $\eta$ on $S^{2n+1}$ by a constant $1/k^2$, then the curvature of the new Riemannian metric will satisfy $\left<\Rm(Y,X)X,Y\right>=\frac{k^2}{4}+\frac{3k^2}{4}\left<X,JY\right>^2$ for all unit tangent vectors $X$ and $Y$ which satisfies $\left<X,Y\right>=0$ with respect to the new Riemannian metric. In order to make everything compatible, one also need to multiply the old Reeb field by $k^2$ to get the new field. The length of an orbit of the new Reeb field becomes $\frac{4\pi}{k^2}$. The complex projective space equipped with this K\"ahler structure is denoted by $\mathcal B(k,1)$. The total space of the Hopf fibration together with contact metric structure induced from $\mathcal B(k,1)$ is denoted by $\mathcal P(k,1)$. Finally, the quotient of $\mathcal P(k,1)$ by the discrete subgroup of $S^1$ of order $m$ is denoted by $\mathcal P(k,m)$.

\smallskip

\section{Myers' Type Maximal Diameter Theorems for Symplectic and K-contact Manifolds}\label{SMyers}

A contact metric manifold is $K$-contact if the Reeb field $\xi$ is Killing. In this section, we assume that the contact manifold is $K$-contact and give the proof of the following Myers' type maximal diameter theorem. This guarantees that all manifolds which satisfy the conditions in Theorem \ref{main} are compact.

\begin{thm}\label{Myer}
Let $\mathcal N:\ker\eta\to\ker\eta$ be the bilinear form defined by $\mathcal N(X)=P(\nabla_X\cJ)X$, where $P:TM\to\ker\eta$ is the orthogonal projection.
\begin{enumerate}
\item Assume that
\begin{equation}\label{bound1}
\left<\overline\Rm(\cJ X,X)X,\cJ X\right>-|\mathcal N(X)|^2\geq k_1^2
\end{equation}
for all unit tangent vectors $X$. Then the diameter of $M$ with respect to the metric $d_{CC}$ is bounded above by $\frac{2\pi}{k_1}$.
\item Assume that
\begin{equation}\label{bound2}
\tr_{\{X,\cJ X\}^\perp}(Y\mapsto\left<\overline\Rm(Y,X)X,Y\right>)-|\mathcal N(X)|^2\geq k_2^2
\end{equation}
for all unit tangent vector $X$, where $\tr_{\{X,\cJ X\}^\perp}$ denotes the trace of the bilinear form defined on $\{X,\cJ X\}^\perp$. Then the diameter of $M$ with respect to the metric $d_{CC}$ is bounded above by $\frac{\sqrt{2n-2}\pi}{k_2}$.
\end{enumerate}
\end{thm}

Note that (\ref{bound1}) becomes a lower bound on a CR analogue of holomorphic sectional curvature $\left<\overline\Rm(\cJ X,X)X,\cJ X\right>$ if the manifold is Sasakian.

Next, we state a result which is a symplectic analogue of Theorem \ref{Myer}. In the following theorem, $M$ is a manifold of dimension $2n$ equipped with a symplectic structure $\omega$, an almost complex structure $\cJ$, and a compatible Riemannian metric $\left<\cdot,\cdot\right>$. In particular, $\left<X,\cJ Y\right>=\omega(X,Y)$.

\begin{thm}\label{symplMyer}
Let $\mathcal N:TM\to TM$ be the bilinear form defined by $\mathcal N(X)=(\nabla_X\cJ)X$.
\begin{enumerate}
\item Assume that
\begin{equation}\label{bound1s}
\left<\Rm(\cJ X,X)X,\cJ X\right>-|\mathcal N(X)|^2\geq k_1^2
\end{equation}
for all unit tangent vectors $X$. Then the diameter of $M$ with respect to the Riemannian metric $d$ is bounded above by $\frac{\pi}{k_1}$.
\item Assume that
\begin{equation}\label{bound2s}
\tr_{\{X,\cJ X\}^\perp}(Y\mapsto\left<\Rm(Y,X)X,Y\right>)-|\mathcal N(X)|^2\geq k_2^2
\end{equation}
for all unit tangent vector $X$, where $\tr_{\{X,\cJ X\}^\perp}$ denotes the trace of the bilinear form defined on $\{X,\cJ X\}^\perp$. Then the diameter of $M$ with respect to the Riemannian metric $d$ is bounded above by $\frac{\sqrt{2n-2}\pi}{k_2}$.
\end{enumerate}
\end{thm}

Note that when the manifold $M$ is K\"ahler, both conditions (\ref{bound1s}) and (\ref{bound2s}) are satisfied if the bisectional curvature is bounded below by a positive constant (see \cite{TaYu} for a closely related result).

\begin{proof}[Proof of Theorem \ref{Myer}]
Since the manifold is $K$-contact, the tensor $h$ vanishes. It also follows that $a(t)$ is independent of time. Using the same notations as in the previous section, we have
\[
\begin{split}
&\dot S_1(t)+L'(t) +S_1(t)C_3S_1(t)-S_1(t)W'(t)^T-W'(t)S_1(t)=0,
\end{split}
\]
where
\[
W'(t)=W(t)-\frac{a}{2}J(t)+\frac{1}{2}C_2
\]
and
\[
L'(t)=R(t)+\frac{a}{2}K_1(t)-\frac{1}{4}C_1 +\frac{a^2}{4}C_3.
\]

Let $S_1(t)=\left(\begin{array}{cc}
S_{1,0}(t) & S_{1,1}(t)\\
S_{1,1}(t)^T & S_{1,2}(t)
\end{array}\right)$, $W'(t)=\left(\begin{array}{cc}
W'_{0} & W'_{1}(t)\\
-W'_{1}(t)^T & W_{2}'(t)
\end{array}\right)$, $L'(t)=\left(\begin{array}{cc}
L_{0}'(t) & L_{1}'(t)\\
L_{1}'(t)^T & L_{2}'(t)
\end{array}\right)$, $K_1(t)=\left(
\begin{array}{cc}
K_{1,0}(t) & K_{1,1}(t)\\
K_{1,1}(t)^T & K_{1,2}(t)
\end{array}\right)$, and $J(t)=\left(\begin{array}{cc}
J_{0} & 0\\
0 & J_{2}(t)
\end{array}\right)$, where $S_{1,0}(t)$, $W_0'(t)$, $L_0'(t)$, $K_{1,0}(t)$, and $J_0$ are $3\times 3$ blocks.

A computation using Theorem \ref{S1}, Proposition \ref{J}, and Proposition \ref{Rm} shows that
\[
\begin{split}
&J_0=\left(\begin{array}{ccc}
0 & 0 & 0\\
0 & 0 & 1\\
0 & -1 & 0
\end{array}\right),\\
&K_{1,0}=\left(\begin{array}{ccc}
0 & -\frac{1}{2} & 0\\
-\frac{1}{2} & 0 & 0\\
0 & 0 & 0
\end{array}\right),\\
&W_0'=\left(
\begin{array}{ccc}
0 & 0 & -\frac{1}{2}\\
0 & 0 & \frac{a}{2}\\
1 & -\frac{a}{2} & 0
\end{array}\right),\\
&W_1'(t)=\left(\begin{array}{ccc}
0 & ... & 0\\
0 & ... & 0\\
N_3(t) & ... & N_{2n}(t)
\end{array}\right)U(t)^T,\\
&W_2'(t)=-\frac{a}{2}J_2(t).
\end{split}
\]

The block $S_{1,0}(t)$ satisfies
\[
\begin{split}
&0=\dot S_{1,0}(t)-W_0'S_{1,0}(t)-W_1'(t)S_{1,1}(t)^T-S_{1,0}(t)^TW_0'^T\\
& -S_{1,1}(t)W_1'(t)^T+S_{1,0}(t)\bar C_3S_{1,0}(t)+S_{1,1}(t)S_{1,1}(t)^T+L_0'(t)\\
&\geq \dot S_{1,0}(t)-W_0'S_{1,0}(t)-S_{1,0}(t)^TW_0'^T\\
& +S_{1,0}(t)\bar C_3S_{1,0}(t)+L_0'(t)-W_1'(t)W_1'(t)^T,
\end{split}
\]
where $L_0'(t)=\left(\begin{array}{ccc}
\frac{1}{4} & -\frac{a}{4} & 0\\
-\frac{a}{4} & \frac{a^2}{4} & 0\\
0 & 0 & \bar R_{22}(t)-1+\frac{a^2}{4}
\end{array}\right)$, \\
$W_1(t)'W_1(t)'^T=\left(\begin{array}{ccc}
0 & 0 & 0\\
0 & 0 & 0\\
0 & 0 & |N(t)|^2
\end{array}\right)$, and \\
$|N(t)|^2=N_3(t)^2+...+N_{2n}(t)^2$.

Here $\overline\Rm$ is the Tanaka-Webster curvature and
\[
\bar R_{22}(t)=\left<\overline\Rm(v_2(t),v_1(t))v_1(t),v_2(t)\right>.
\]

Under the assumptions of the theorem, we have
\[
\bar R_{22}(t)-|N(t)|^2\geq k_1^2.
\]

Solutions of the comparison equation
\[
\begin{split}
&0=\dot{\bar S}_{0}(t)-W_0'\bar S_{0}(t)-\bar S_{0}(t)^TW_0'^T +\bar S_{0}(t)\bar C_3\bar S_{0}(t)+\bar L_0'
\end{split}
\]
can be computed explicitly using the method in \cite{Lev}. Here
\[
\bar L_0'(t)=\left(\begin{array}{ccc}
\frac{1}{4} & -\frac{a}{4} & 0\\
-\frac{a}{4} & \frac{a^2}{4} & 0\\
0 & 0 & k_1^2-1+\frac{a^2}{4}
\end{array}\right).
\]

A solution is given by
\[
\bar S_0(t)=\left(
\begin{array}{ccc}
\frac{2a^2-2a^2\cos(c_1t)+k_1^2c_1t\sin(c_1t)}{ts(t)} & \frac{a}{t} & \frac{2c_1^2-2+2(1-c_1^2)\cos(c_1t)+c_1t\sin(c_1t)}{2s(t)}\\
\frac{a}{t} & \frac{1}{t} & \frac{a}{2}\\
\frac{2c_1^2-2+2(1-c_1^2)\cos(c_1t)+c_1t\sin(c_1t)}{2s(t)} & \frac{a}{2} & \frac{c_1(\sin(c_1t)-c_1t\cos(c_1t))}{s(t)}
\end{array}\right),
\]
where $s(t)=2-2\cos(c_1t)-c_1t\sin(c_1t)$ and $c_1=\sqrt{k_1^2+a^2}$.

As $t\to 0$, it grows like
\[
\left(
\begin{array}{ccc}
\frac{12}{t^3} & \frac{a}{t} & \frac{6}{t^2}\\
\frac{a}{t} & \frac{1}{t} & \frac{a}{2}\\
\frac{6}{t^2} & \frac{a}{2} & \frac{4}{t}
\end{array}\right),
\]
so $\lim_{t\to 0}\bar S_0(t)\to \infty$ as $t\to 0$.

By the comparison theorem \cite{Ro} of matrix Riccati equations, $\bar S_0(t)\geq S_{1,0}(t)$. Note that $S_{1,0}(0)$ is the matrix defined by $-\nabla^2g(\gamma(t_0))$ which is well-defined for all small positive $t_0$ (recall that $g$ is $C^\infty$ along $\gamma$ except at the end-points). This approach does not require any short time asymptotic. It follows that the geodesic is no longer minimizing if $t>\frac{2\pi}{k_1}$. Therefore, the diameter of $M$ is less than or equal to $\frac{2\pi}{k_1}$.

Next, we look at the equation satisfied by $S_{1,2}(t)$.
\[
\begin{split}
&0=\dot S_{1,2}(t)+L'_2(t)+S_{1,1}(t)^T\bar C_3S_{1,1}(t)+S_{1,2}(t)^2+W_1'(t)^TS_{1,1}(t)\\
&-W_2'(t)^TS_{1,2}(t)+S_{1,1}(t)^TW_1'(t)-S_{1,2}(t)^TW_2'(t)\\
&=\dot S_{1,2}(t)+L'_2(t)+(S_{1,1}(t)+W_1'(t))^T\bar C_3(S_{1,1}(t)+W_1'(t))+S_{1,2}(t)^2\\
&-W_1'(t)^TW_1'(t)-W_2'(t)^TS_{1,2}(t)-S_{1,2}(t)^TW_2'(t)\\
&\geq\dot S_{1,2}(t)+S_{1,2}(t)^2 +L'_2(t)-W_1'(t)^TW_1'(t) +\frac{a}{2}J_2^TS_{1,2}(t)+\frac{a}{2}S_{1,2}(t)^TJ_2.
\end{split}
\]

Let $s(t)=\tr(S_{1,2}(t))$. After taking the trace, we obtain
\[
\begin{split}
&0\geq\dot s(t)+\frac{1}{2n-2}s(t)^2+\tr(L_2'(t)-W_1'(t)^TW_1'(t))\\
&\geq\dot s(t)+\frac{1}{2n-2}s(t)^2+k_2^2.
\end{split}
\]

Let $\bar s(t)$ be
\[
\bar s(t)=\sqrt{2n-2}\,k_2\cot\left(\frac{k_2t}{\sqrt{2n-2}}\right)
\]
It is a solution of the comparison equation $\dot{\bar s}(t)+\frac{1}{2n-2}\bar s(t)^2+k_2^2$. By the same comparison principle as above, $s(t)\leq \bar s(t)$. It follows that the diameter is bounded by $\frac{\sqrt{2n-2}\pi}{k_2}$ in this case.
\end{proof}

\begin{proof}[Proof of Theorem \ref{symplMyer}]
The proof is similar to and simpler than the one for Theorem \ref{Myer}. We give a very brief sketch.

Let $\gamma:[0,d(x,x_0)]\to M$ be a geodesic which starts from $x$ and ends at $x_0$. Let us fix an orthonormal frame $\{v_1(t),...,v_{2n}(t)\}$ defined along $\gamma$ such that $v_1(t)=\dot\gamma(t)$, $v_2(t)=\cJ\dot\gamma(t)$, and $\dot v_i(t)\in\text{span}\{v_1(t),v_2(t)\}$ for all $i=3,...,2n$.

Let $v(t)=\left(\begin{array}{ccc} v_1(t) & ... & v_{2n}(t)\end{array}\right)^T$. A computation shows that
\[
\dot v(t)=\left(\begin{array}{cc}
0 & A_1(t)\\
-A_1(t)^T & 0
\end{array}\right)v(t)
\]
where $A_1(t)=\left(\begin{array}{ccc}
0 & ... & 0\\
N_3(t) & ... & N_{2n}(t)
\end{array}\right)$ and $N_i(t)=\left<(\nabla_{\dot\gamma(t)}\cJ)\dot\gamma(t), v_i(t)\right>$.

The curve $\gamma$ satisfies $\dot\gamma(t)=-\nabla g(\gamma(t))$, where $g(x)=d(x_0,x)$. The function $g$ satisfies $|\nabla g|=1$ and a computation shows that $S_{ij}(t)=-\nabla^2g_{\gamma(t)}(v_i(t),v_j(t))$ satisfies
\[
\dot S(t)+S(t)^2-A(t)S(t)-S(t)A(t)^T+R(t)=0.
\]

Let $S(t)=\left(\begin{array}{cc}
S_0(t) & S_1(t)\\
S_1(t)^T & S_2(t)
\end{array}\right)$ and $R(t)=\left(\begin{array}{cc}
R_0(t) & R_1(t)\\
R_1(t)^T & R_2(t)
\end{array}\right)$, where $S_0(t)$ and $R_0(t)$ are $2\times 2$ blocks. The blocks $S_0(t)$ and $S_2(t)$ satisfy
\[
\begin{split}
&0=\dot S_0(t)+S_0(t)^2+S_1(t)S_1(t)^T-A_1(t)S_1(t)^T-S_1(t)A_1(t)^T+R_0(t)\\
&\geq \dot S_0(t)+S_0(t)^2-A_1(t)A_1(t)^T+R_0(t).
\end{split}
\]
and
\[
\begin{split}
&0\geq\dot s(t) +\frac{1}{2n-2}s(t)^2+\tr(-A_1(t)^TA_1(t)+R_2(t)).
\end{split}
\]
where $s(t)=\tr(S_2(t))$.

The results follow as in the proof of Theorem \ref{Myer}.

\end{proof}

\smallskip

\section{Comparison Theorems for the Closed Reeb Orbit}\label{SClosedOrbit}

In this section, we prove two comparison type theorems in the spirit of \cite{HeKa} for the closed Reeb orbit. First, a Myers' type result.

\begin{thm}\label{orbitd}
Let $C$ be a closed orbit of $\xi$ and let $\gamma:[0,T]\to M$ be a unit speed minimizing geodesic which starts from a point $x$ and ends at a point in $C$ which is closest to $x$.

\begin{enumerate}
\item Assume that (\ref{bound1}) holds. Then $T\leq \frac{\pi}{k_1}$.
\item Assume that (\ref{bound2}) holds. Then $T\leq \frac{\sqrt{2n-2}\pi}{2k_2}$.
\end{enumerate}
\end{thm}

The second one is a volume growth estimate.

\begin{thm}\label{HK}
Assume that conditions (\ref{bound1}) and (\ref{bound2}) hold with $k_2=\frac{\sqrt{2n-2}k_1}{2}$. Let $C$ be a closed orbit of $\xi$. Let $V(C,T)$ be the neighborhood of $C$ of radius $T$
\[
V(C,T):=\vol(\{x\in M|d_{CC}(C,x)<T\}).
\]
Let $\bar V(T)$ be the corresponding volume in the model $\mathcal P(k_1,1)$. Then
$\frac{V(C,T)}{\bar V(T)}$ is non-increasing in $T$.
\end{thm}

\begin{proof}[Proof of Theorem \ref{orbitd}]
Let $\psi:U:=\{(x,v)\in\ker\eta_C|\,|v|=1\}\to M$ be the map defined as the solution of the initial value problem:
\[
\frac{D^2}{dt^2} \psi(x,tv)=0,\quad \psi(x,0)=x,\quad \frac{D}{dt} \psi(x,tv)\Big|_{t=0}=v.
\]
i.e. the restriction of the exponential map to the distribution (in this case the exponential map can be the Riemannian or the sub-Riemannian one).

For each $(x,v)$ in $U$, let $e_1(0)=\xi(x)$ and $e_2(0)=\cJ v$. Let $e_i(t)$ be defined as above along $\psi(x,tv)=:\gamma(t)$ such that $\dot e_i(t)$ is contained in $\text{span}\{e_1(t),e_2(t)\}$, where $i=3,...,2n$. Let $E(t)=\left(\begin{array}{ccccc}
e_1(t) & ... & e_{2n}(t)
\end{array}\right)^T$ and let $A(t)$ be a family of matrices defined by $\dot E(t)=A(t)E(t)$. A computation shows that
\[
\begin{split}
&\dot e_1(t)=\nabla_{\dot\gamma}\xi=-\frac{1}{2}e_2(t),\\
&\dot e_2(t)= (\nabla_{\dot\gamma}\cJ)\dot\gamma(t) =\frac{1}{2}e_0(t)+\sum_{i\geq 3}N_i(t)e_i(t),
\end{split}
\]
where $N_i(t)=\left<(\nabla_{e_1(t)}\cJ)e_1(t),e_i(t)\right>$.

Hence, $A(t)=\left(\begin{array}{cc}
A_0(t) & A_1(t)\\
-A_1(t)^T & 0
\end{array}\right)$, where $A_0(t)=\left(\begin{array}{cc}
0 & -\frac{1}{2}\\
\frac{1}{2} & 0
\end{array}\right)$, $A_1(t)=\left(\begin{array}{c}
0\\
N'
\end{array}\right)$, and $N'=\left(\begin{array}{ccc}N_3 & ... & N_{2n}\end{array}\right)$.

For $i\neq 1$, let $\sigma_i$ be the curve defined by $\sigma_i(0)=tv$ and $\bar e_i(t)=\dot\sigma_i(0)$, where $\bar e_i(t)$ is the vector induced by $e_i(0)$. For $i=1$, let $\sigma_1(s)=d\Phi_s(v)$. Let $d\psi_{(x,tv)}(\bar E(t))=B(t)E(t)$.

When $t=0$, $B(0)=I$. Since $\Phi_s$ is an isometry, $\Phi_s(\gamma)$ is also a geodesic if $\gamma$ is. It follows that $\psi(d\Phi_s(tv))=\Phi_s(\psi(tv))$. Therefore,
\[
\sum_iB_{1i}(t)e_i(t)=d\psi_{(x,tv)}(\bar e_1(t))=\frac{D}{ds}\psi(d\Phi_s(tv))\Big|_{s=0}=\xi(\psi(tv))
\]
So $B_{1i}(t)=\delta_{1i}$.

Let $t\mapsto\varphi_t(x)$ be the geodesic connecting from $x$ to the point in closed orbit $C$ which is closest to $x$. It satisfies $\dot\varphi_t(x) =\nabla_Hg(\varphi_t(x))=\nabla g(\varphi_t(x))$, where $g(x)=-d(C,x)$. Note that the curve $\psi(x,tv)$ coincides with $\varphi_{T-t}(\psi(x,Tv))$ if $t\in [0,T]\mapsto\psi(x,tv)$ is minimizing.

For $i\neq 1$,
\[
\begin{split}
&\frac{D}{dt}d\psi_{(x,tv)}(t\bar e_i(t)) =-T\nabla^2f(d\varphi_{T-t}(d\psi_{(x,Tv)}(\bar e_i(0))))\\
&=-T\sum_{j,k}B_{ij}(T)C_{jk}(T-t)\nabla^2g(e_k(t))
\end{split}
\]
and
\[
\begin{split}
&\sum_j(B_{ij}(t)e_j(t) +t\dot B_{ij}(t)e_j(t) +t\sum_kB_{ij}(t)A_{jk}(t)e_k(t))\\
&=\sum_j\frac{D}{dt}(tB_{ij}(t)e_j(t))\\
&=-T\sum_{j,k}B_{ij}(T)C_{jk}(T-t)\nabla^2g(e_k(t))
 \end{split}
\]

By setting $t=T$, it follows that
\[
-D(T)^{-1}\dot D(T)-A(T)=TF(0)
\]
is symmetric, where $D(t)=\left(\begin{array}{c}
B_1(t)\\
tB_2(t)\\
\vdots\\
tB_{n}(t)
\end{array}\right)$.

For $i\neq 1$, we have
\[
\begin{split}
&0=\frac{D}{ds}\frac{D}{dt}\frac{d}{dt}\psi(x,t\sigma_i(s))\Big|_{s=0}\\
&=\frac{D}{dt}\frac{D}{dt}d\psi_{(x,tv)}(t\bar e_i(0)) +\Rm(d\psi_{(x,tv)}(t\bar e_i(0)),\dot\gamma)\dot\gamma\\
&=\sum_j\frac{D}{dt}\left(\dot D_{ij}(t)e_j(t)+D_{ij}(t)\sum_kA_{jk}(t)e_k(t)\right) +\sum_{k,l}D_{ik}(t)R_{kl}(t)e_l(t)\\
&=\sum_j\Big(\ddot D_{ij}(t)e_j(t)+2\dot D_{ij}(t)\sum_kA_{jk}(t)e_k(t)\\
& + D_{ij}(t)\sum_k(\dot A_{jk}(t)+\sum_lA_{jl}A_{lk})e_k(t)\Big) +\sum_{k,l}D_{ik}(t)R_{kl}(t)e_l(t).
\end{split}
\]

For $i=1$, let $\Phi_s$ be the flow of $\xi$.
\[
\begin{split}
&0=\frac{D}{ds}\frac{D}{dt}\frac{d}{dt}\psi(d\Phi_s(tv))\Big|_{s=0}\\
&=\frac{D}{dt}\frac{D}{dt}d\psi_{(x,tv)}(\bar e_0(t)) +\Rm(d\psi_{(x,tv)}(\bar e_0(t)),\dot\gamma)\dot\gamma\\
&=\sum_j\frac{D}{dt}(\dot D_{0j}(t)e_j(t)+D_{0j}(t)\sum_kA_{jk}(t)e_k(t)) +\sum_{k,l}D_{0k}(t)R_{kl}(t)e_l(t)\\
&=\sum_j\Big(\ddot D_{0j}(t)e_j(t)+2\dot D_{0j}(t)\sum_kA_{jk}(t)e_k(t)\\
& + D_{0j}(t)\sum_k(\dot A_{jk}(t)+\sum_lA_{jl}A_{lk})e_k(t)\Big) +\sum_{k,l}D_{0k}(t)R_{kl}(t)e_l(t).
\end{split}
\]

By combining the above two equations, we obtain
\[
D(t)^{-1}\ddot D(t)+2D(t)^{-1}\dot D(t)A(t)+\dot A(t)+A(t)^2+R(t)=0.
\]

Let $T(t)=D(t)^{-1}\dot D(t)$ and $S(t)=T(t)+A(t)$. The matrix $S(t)$ satisfies
\begin{equation}\label{mriccati}
\dot S(t)+S(t)^2+S(t)A(t)+A(t)^TS(t)+R(t)=0.
\end{equation}

Let $S(t)=\left(
\begin{array}{cc}
S_0(t) & S_1(t)\\
S_1(t)^T & S_2(t)
\end{array}\right)$, where $S_0(t)$ is a $2\times 2$ block.
\begin{equation}\label{riccati1}
\begin{split}
&0=\dot S_0(t)+S_0(t)^2+S_1(t)S_1(t)^T+S_0(t)A_0(t)\\
&-S_1(t)A_1(t)^T +A_0(t)^TS_0(t)-A_1(t)S_1(t)^T+R_0(t)\\
&=\dot S_0(t) +S_0(t)^2+(S_1(t)-A_1(t))(S_1(t)-A_1(t))^T\\ &+S_0(t)A_0(t)+A_0(t)^TS_0(t)-A_1(t)A_1(t)^T+R_0(t)\\
&\geq\dot S_0(t)+S_0(t)^2+S_0(t)A_0(t)\\
&+A_0(t)^TS_0(t)-A_1(t)A_1(t)^T+R_0(t).
\end{split}
\end{equation}

Let us split $S_0(t)$ further $S_0(t)=\left(
\begin{array}{cc}
S_{0,0}(t) & S_{0,1}(t)\\
S_{0,1}(t) & S_{0,2}(t)
\end{array}
\right)$. A computation shows that $S_{0,0}=0$, $S_{0,1}=-1/2$, and $A_{0,1}=-1/2$.

The function $S_{0,2}(t)$ satisfies
\begin{equation}\label{riccati2}
\begin{split}
&0\geq\dot S_{0,2}(t)+S_{0,1}(t)^2 +S_{0,2}(t)^2\\
&+2S_{0,1}(t)A_{0,1}(t)-|N'(t)|^2+R_{22}(t)\\
&=\dot S_{0,2}(t)+S_{0,2}(t)^2-|N'(t)|^2+\bar R_{22}(t)\\
&\geq\dot S_{0,2}(t) +S_{0,2}(t)^2+k_1^2.
\end{split}
\end{equation}

A solution of the comparison equation $\dot{\bar S}_{0,2}(t)+\bar S_{0,2}(t)^2+k_1^2=0$ with the condition $\lim_{t\to 0}\bar S_{0,2}(t)=\infty$ is given by
\[
\bar S_{0,2}(t)=k_1\cot(k_1t).
\]
By the comparison principle of Riccati equations \cite{Ro}, $S_{0,2}(t)\leq\bar S_{0,2}(t)$. Since $g$ is $C^\infty$ along minimizing geodesics except at the endpoints, it follows that $T\leq \frac{\pi}{k_1}$.

Similarly, the block $S_2(t)$ satisfies
\begin{equation}\label{riccati3}
\begin{split}
&0=\dot S_2(t)+(S_1(t)+A_1(t))^T(S_1(t)+A_1(t))\\
&+S_2(t)^2-A_1(t)^TA_1(t)+R_2(t)\\
&\geq \dot S_2(t)+S_2(t)^2-A_1(t)^TA_1(t)+R_2(t).
\end{split}
\end{equation}
Let $s(t)=\tr(S_2)$. It follows that
\begin{equation}\label{riccati4}
\begin{split}
&0\geq \dot s(t)+\frac{1}{2n-2}s(t)^2+k_2^2.
\end{split}
\end{equation}

A solution to the comparison equation $\dot s+\frac{1}{2n-2}s^2+k_2^2=0$ which satisfies the condition $\lim_{t\to 0}\bar s(t)=\infty$ is given by
\[
\bar s(t)=\sqrt{2n-2}k_2\cot(k_2t/\sqrt{2n-2}).
\]
Once again, by the comparison principle, we have $\tr(S(t))\leq\tr(\bar S_0(t))+s(t)=\tr(\bar S(t))$. Therefore, $T\leq\frac{\sqrt{2n-2}\pi}{k_2}$.
\end{proof}

\begin{proof}[Proof of Theorem \ref{HK}]
Let $U=\{(x,v)\in\ker\eta|x\in C\text{ and } |v|=1\}$. Below, we denote the geometric object in the case of the model $\mathbb{CP}^n$ by adding a bar above the symbol. For instance $\bar U$ denotes the set $U$ in the case of the model.

Let $r:U\to\Real$ be the first time for which the curve $t\mapsto\psi(x,tv)$ is no longer minimizing.
\[
\begin{split}
&V(C,T):=\int_{\psi(\{(x,tv)|(x,v)\in U\text{ and }t\in[0,\min\{T,r(x,v)\}]\})}\vol\\
&=\int_U\int_0^{\min\{T,r(x,v)\}}\det(B(t))dt\\
&=\int_U\int_0^{T}\det(B(t))dt.
\end{split}
\]
Here we extend $\det(B(t))$ by zero when $t>r(x,v)$.

By using the fact that $B_{1i}(t)=\delta_{1i}$ and $\tr(S)\leq\tr(\bar S_0)+s=\tr(\bar S)$,
\[
\frac{d}{dt}\frac{\det(B(t))}{\det(\bar B(t))}=\frac{\det(B(t))\det(\bar B(t))(\tr(S)-\tr(\bar S))}{\det(\bar B(t))^2}\leq 0,
\]
where $\bar B(t)$ is $B(t)$ in the case when $M$ is the complex projective space $\mathbb{CP}^n$.

It follows that $\frac{\det(B(t))}{\det(\bar B(t))}$ is non-increasing. This also holds when $t>r(x,v)$. The average with respect to $\det(\bar B(t))dt$ is also non-increasing. It follows that
\[
T\mapsto\frac{\vol{(\bar U)}\int_U\int_0^T\det(B(t))dt}{\vol{(U)}\int_{\bar U}\int_0^T\det(\bar B(t))dt}=\frac{\vol(\bar U)V(C,T)}{\vol(U)\bar V(T)}
\]
is also non-increasing.
\end{proof}

\smallskip

\section{The Equality case}\label{SEqual}

By Theorem \ref{HK},
\[
\frac{V(C,\frac{\pi}{k_1})}{\text{len}(C)}\leq \frac{\bar V(\frac{\pi}{k_1})}{\text{len}(\bar C)},
\]
where $\text{len}(C)$ and $\text{len}(\bar C)$ are the lengths of closed orbits $C$ and $\bar C$ in $M$ and the model $\mathcal P(k_1,1)$, respectively. This section is devoted to the proof of the following key lemma which deals with the case when the above inequality becomes an equality.

\begin{lem}\label{key}
Let $T=\frac{\pi}{k_1}$. Assume that $m:=\frac{\bar V(T)}{V(C,T)}=\frac{\text{len}(\bar C)}{\text{len}(C)}$. Let $X$ be the set of points in $M$ which are of distance $\frac{\pi}{k_1}$ away from $C$. Then
\begin{enumerate}
    \item $m$ is a positive integer,
    \item $X$ is a totally geodesic submanifold of $M$,
    \item the tangent bundle $TX$ of $X$ is invariant under the tensor $\cJ$,
    \item the Reeb field is tangent to $X$,
    \item the exponential map is a $m$-fold covering from the set $P_x:=\{(x,v)\in\ker\eta|\,|v|=\pi/k_1\}$ to $X$ for each $x$ in $C$,
    \item the $(2n-1)$-dimensional volume of $X$ is equal to that of the corresponding submanifold $\bar X$ in the model $\mathcal P(k_1,m)$.
\end{enumerate}
\end{lem}

\begin{proof}
The equality $\frac{V(C,T)}{\text{len}(C)} =\frac{\bar V(T)}{\text{len}(\bar C)}$ implies that all inequalities including (\ref{riccati1}), (\ref{riccati2}), (\ref{riccati3}), and (\ref{riccati4}) become equality and the coefficients equal to that of the model case.
First, we have $r(x,v)=\frac{\pi}{k_1}$. It follows from (\ref{riccati1}) that $S_1(t)=A_1(t)$, from (\ref{riccati2}) that $N'=0$ and $\bar R_{22}(t)=k_1^2$ and $S_{0,2}(t)=k_1\cot(k_1t)$, from (\ref{riccati3}) that $S_1(t)=-A_1(t)=0$, and from (\ref{riccati4}) that $S_2(t)=c\cot(ct)I$, where $c=\frac{k_2}{\sqrt{2n-2}}=\frac{k_1}{2}$.

The above arguments show that $S(t)$ and $A(t)$ coincides with the corresponding quantities in the model case. By substituting this into (\ref{mriccati}), it follows that the same holds for $R(t)$. Since $B(t)$ satisfies the same equations and initial condition as the corresponding quantity in the model case, the two quantities coincide as well. In summary, we obtain the followings:
\begin{equation}\label{SAB}
\begin{split}
&S(t)=\left(
\begin{array}{ccc}
0 & -1/2 & 0\\
-1/2 & k_1\cot(k_1t) & 0\\
0 & 0 & c\cot(ct)I
\end{array}\right), \\
&A(t)=\left(
\begin{array}{ccc}
0 & -1/2 & 0\\
1/2 & 0 & 0\\
0 & 0 & 0_{(2n-1)\times(2n-1)}
\end{array}\right)\\
&B(t)=\left(
\begin{array}{ccc}
1 & 0 & 0\\
-\frac{1-\cos(k_1t)}{k_1^2t} & \frac{\sin(k_1t)}{k_1t} & 0\\
0 & 0 & \frac{\sin(ct)}{ct}I
\end{array}\right)
\end{split}
\end{equation}

Therefore,
\begin{equation}\label{D}
D(T)=\left(
\begin{array}{ccc}
1 & 0 & 0\\
-\frac{2}{k_1^2} & 0 & 0\\
0 & 0 & \frac{2}{k_1}I
\end{array}\right)
\end{equation}

\begin{equation}\label{dD}
\dot D(T)=\left(
\begin{array}{ccc}
0 & 0 & 0\\
0 & -1 & 0\\
0 & 0 & 0
\end{array}\right).
\end{equation}

It follows that the derivative of the exponential map restricted to $P_x$ has full rank and so it is an immersion.
The second row of $D(T)$ in (\ref{D}) shows that the derivative of the exponential map sends $e_1(0)$ to the Reeb field $\xi$. Therefore, $\xi$ is tangent to the submanifold $X$. Since the second column of $D(T)$ vanishes and the vector field $e_2(T)$ is orthogonal to $X$. The third and the fourth assertions follow.

Let $(\Phi_\tau(x),d\Phi_\tau(v))$ be another vector in $P_{\Phi_\tau(x)}$. Let $\gamma$ be the geodesic connecting $x$ to $X$. Since $\Phi_\tau$ is distance preserving and $\xi$ is tangent to $X$, $\Phi_\tau\circ\gamma$ has the same length as $\gamma$ and it connects $\Phi_\tau(x)$ to $X$. It follows that the image of $P_{\Phi_\tau(x)}$ under the exponential map is contained in that of $P_x$. By symmetry, they are the same.

Let $v(t)=d\Phi_t(v)$. Since $\nabla_\xi\cJ=0$. It follows that $\cJ v(t)=d\Phi_t(\cJ v)$. Indeed,
\[
\begin{split}
&\frac{D}{dt}v(t)=\frac{D}{ds}\xi(\Phi_t(\gamma(s)))\Big|_{s=0}=-\frac{1}{2}\cJ(d\Phi_t(v)),
\end{split}
\]
where $\dot\gamma(0)=v$.

It follows that
\[
\frac{D}{dt}\left(\cJ v(t)\right)=\frac{1}{2}d\Phi_t(v)=\frac{D}{dt}d\Phi_t(\cJ v)
\]
and the claim follows from this and that two vector fields coincide at $t=0$.

Let $w(t) = \frac{\pi}{k_1}\left(\cos(ct)d\Phi_t(v)+\sin(ct)\cJ d\Phi_t(v)\right)$. A computation shows that
\[
\begin{split}
\frac{d}{dt}w(t)\Big|_{t=t_0}&=\frac{d}{dt}\left(\frac{\pi}{k_1}\left(\cos(ct_0)d\Phi_t(v)+\sin(ct_0)\cJ d\Phi_t(v)\right)\right)\\
&+\frac{c\pi}{k_1}\left(-\sin(ct_0)d\Phi_{t_0}(v)+\cos(ct_0)\cJ d\Phi_{t_0}(v)\right)
\end{split}
\]

The first term on the right equals to $\bar e_1(0)$ and the second one equals to $c\bar e_2(0)$ defined at $w(t_0)$. This discussion together with (\ref{D}) gives
\[
\begin{split}
&\frac{d}{dt}\exp\left(w(t)\right)\Big|_{t=t_0} =d\exp_{w(t)}\left(\frac{d}{dt}w(t)\right)\Big|_{t=t_0}=0.
\end{split}
\]
if $c=\frac{k_1^2}{2}$.

This shows that $\exp(w(t))$ is a point in $X$ independent of time $t$. Moreover, by (\ref{dD}),
\[
\begin{split}
&\frac{D}{ds}\frac{d}{dt}\exp(t w(s))\Big|_{t=T,s=0} =\frac{D}{dt}\frac{d}{ds}\exp(t w(s))\Big|_{t=T,s=0}=-c e_2(0).
\end{split}
\]
Therefore, the tangent vectors $\frac{d}{dt}\exp(t w(s))\Big|_{t=T}$ is rotating at speed $c$ in a circle in the normal bundle of $X$ at $\exp(w(t))$. Since the speed of rotation is $c$, it takes time $\frac{4\pi}{k_1^2}=\text{len}(\bar C)$ to rotate around once.

On the other hand, since
\[
\frac{d}{dt}\exp(t w(0))\Big|_{t=T}=\frac{d}{dt}\exp(t w(4\pi/k_1^2))\Big|_{t=T},
\]
$w(4\pi/k_1^2)=w(0)$ by the uniqueness of the geodesic. It follows that $m$ is the number of time the curve $s\in[0,4\pi/k_1^2)\mapsto\Phi_s(x)$ hits the point $x$ and so it is a positive integer proving the first statement. By continuity, this integer $m$ is the same for each $v$ in the normal bundle. Therefore, the exponential map restricted to $P_x$ is a $m$-fold covering of the submanifold $X$.

Since the volume of $X$ can be written in terms of $B(T)$ and $m$, it equals to the corresponding volume in the model space $\mathcal P(k_1,m)$. This is the last assertion.

Finally, the fact that $X$ is totally geodesic follows from (\ref{dD}). Indeed, from the above discussion, any normal vector of $X$ is of the form $\frac{d}{dt}\exp(t v)\Big|_{t=T}$, where $v$ is any unit vector in the normal bundle of $C$. The shape operator is given by
\[
\frac{D}{ds}\frac{d}{dt}\exp(t v(s))\Big|_{t=T,s=0}=\frac{D}{dt}\frac{d}{ds}\exp(t v(s))\Big|_{t=T,s=0}.
\]
The above quantity is essentially given by the lower right diagonal block of $\dot D(T)$ which vanishes. So the shape operator is zero everywhere and $X$ is totally geodesic.
\end{proof}

\smallskip

\section{Proof of Theorem \ref{main}}\label{endproof}

By Lemma \ref{key}, the submanifold $X$ is totally geodesic submanifold of $M$, the Reeb field $\xi$ is tangent to $X$, and the tangent bundle $TX$ is invariant under $\cJ$. It follows that all the geometric structures restrict to $X$ and give it a $K$-contact manifold structure of dimensions two lower than that of $M$. Moreover, since $X$ is totally geodesic, the curvature and hence the conditions in Theorem \ref{main} is preserved. Therefore, by induction $X$ is isometrically contactomorphic to the model $\mathcal P(k_1,m)$.
The exponential map restricted to the subset
\[
\{(x,v)\in\ker\eta||v|<\pi/k_1\}
\]
of the normal bundle of $C$ defines a diffeomorphism onto $M-X$. This map together with the corresponding exponential map of the model space $\mathcal P(k_1,m)$ defines a diffeomorphism $\Psi_1$ from $M-X$ to $\mathcal P(k_1,m)-\bar X$. By (\ref{SAB}), the matrix $B(t)$, which is the matrix representation of the derivative of the exponential map with respect to the orthonormal moving frames, agrees with the corresponding one in the model. Therefore, $\Psi_1$ is an isometry. The first row of $B(t)$ shows that $\Psi_1$ sends the Reeb field on $M$ to the Reeb field on the model, so it is a contactomorphism.

Finally, by induction (see below for the argument in the three dimensional case), both submanifolds $X$ and $\bar X$ are isometrically contactomorphic to $\mathcal P(k_1,m)$ of dimension $2n-1$. It follows that the exponential maps restricted to the normal bundles of $X$ and $\bar X$ defines a isometric contactomophism $\Psi_2$ from $M-C$ to $\mathcal P(k_1,m)-\bar C$. The analysis in the proof of Lemma \ref{key} shows that the $\Psi_1$ and $\Psi_2$ paste together to form a map from $M$ to $\mathcal P(k_1,m)$.

It remains to consider the three dimensional case. In this case $X$ and $\bar X$ are closed Reeb orbits with the length of the later one equals to $m$-times of that of the former one. The same argument used above works. This finishes the proof.

\smallskip

\section{Appendix}

In this section, we recall several known formulas which are needed for this paper. The proof of them can be found in \cite{Bl} though the formulas here have sightly different constants since there are differences in the notations.

Let us first recall that $M$ is a contact metric manifold. It means that $M$ is equipped with tensors $J, \left<\cdot,\cdot\right>, \eta, \xi$ such that $\eta(\xi)=1$, $d\eta(\xi,\cdot)=0$, $\left<\xi,X_1\right>=0$, $\left<\cJ X_1,\cJ X_2\right>=\left<X_1,X_2\right>$, and $\left<X_1,\cJ X_2\right>=d\eta(X_1,X_2)$ for all tangent vectors $X_1$ and $X_2$ in $\ker\eta$.

Recall that the Nijenhuis tensor is defined by
\[
[\cJ,\cJ](u,v)=\cJ^2[u,v]+[\cJ u,\cJ v]-\cJ[\cJ u,v]-\cJ[u,\cJ v].
\]

\begin{prop}\label{J}
The followings hold for all $u,v,w$ in $TM$.
\begin{enumerate}
\item $\nabla_\xi\xi=0$,
\item $[\cJ,\cJ](v,w) =(\nabla_{\cJ v}\cJ)w-(\nabla_{\cJ w}\cJ)v+\cJ(\nabla_w\cJ)v-\cJ(\nabla_v\cJ)w$,
\item $\left<(\nabla_u \cJ) w, v\right>+\left<(\nabla_v \cJ) u, w\right>+\left<(\nabla_w \cJ) v, u\right>=0$,
\item $2\left<(\nabla_u\cJ)v,w\right>=\left<[\cJ,\cJ](v,w),\cJ u\right>+d\eta(\cJ v, u)\eta(w)-d\eta(\cJ w,u)\eta(v)$,
\item $\nabla_\xi \cJ=0$,
\end{enumerate}
where $\nabla$ denotes the Levi-Civita connection.
\end{prop}

\begin{prop}\label{h}
Assume that $M$ is of dimension $2n+1$. Let $h=\LD_\xi \cJ$. The followings hold for all $u,v,w$ in $TM$.
\begin{enumerate}
\item $\LD_{\cJ u}\eta(v)=\LD_{\cJ v}\eta(u)$,
\item $h\xi=0$,
\item $\left<hu, v\right>=\left<hv, u\right>$,
\item $\nabla_u\xi=-\frac{1}{2}(\cJ u+\cJ hu)$,
\item $\cJ h+h\cJ=0$,
\item $\tr(h)=0$,
\item $\cJ\Rm(\xi,u)\xi=\frac{1}{2}(\nabla_\xi h)u-\frac{1}{4}\cJ u+\frac{1}{4}h^2\cJ u$,
\item $\Rm(\xi, u)\xi -\cJ\Rm(\xi, \cJ u)\xi=\frac{1}{2}\cJ^2u +\frac{1}{2}h^2u$,
\item $\ric(\xi,\xi)=\frac{n}{2} -\frac{1}{4}\tr(h^2)$.
\end{enumerate}
\end{prop}

\begin{prop}\label{Rm}
The followings hold.
\begin{enumerate}
\item $2\left<(\nabla_v\cJ)w, u\right>+2\left<(\nabla_{\cJ v}\cJ)w, \cJ u\right>\\
=\eta(u)\left<w,v+hv\right>-2\eta(w)\left<u,v\right>+\eta(u)\eta(v)\eta(w)$,
\item $\left<(\nabla_v\cJ)v, \xi\right>=\frac{1}{2}\left<v,v+hv\right>-\left<\xi,v\right>^2$,
\item $\Rm(v,u)\xi=-\frac{1}{2}(\nabla_v\cJ)(u)-\frac{1}{2}\nabla_v(\cJ h)(u)+\frac{1}{2}(\nabla_u\cJ)(v)+\frac{1}{2}\nabla_u(\cJ h)v$,
\item $\left<\Rm(\xi, w)v,u\right>-\left<\Rm(\xi, w)\cJ v,\cJ u\right>\\
+\left<\Rm(\xi, \cJ w) \cJ v,u\right>+\left<\Rm(\xi, \cJ w)v, \cJ u\right>\\
=\frac{1}{2}\eta(u)\left<v,w+hw\right>-\frac{1}{2}\eta(v)\left<u,w+hw\right> +\left<(\nabla_{hw}\cJ)u,v\right>$.
\end{enumerate}
\end{prop}

Let $\TW$ be the generalized Tanaka connection defined by
\[
\TW_uv=\nabla_uv+\frac{1}{2}\eta(u)\cJ v-\eta(v)\nabla_u\xi+\left<\nabla_u\xi,v\right>\xi.
\]
Let $T$ and $\overline{\Rm}$ be the torsion and the curvature of $\TW$, respectively.

\begin{prop}\label{bRm}
The followings hold for all $u,v$ in $TM$ and all $X,Y,Z$ in $\ker\eta$.
\begin{enumerate}
\item $T(u,v)=\frac{1}{2}\eta(v)\cJ hu-\frac{1}{2}\eta(u)\cJ hv+g(u,\cJ v)\xi$,
\item $T(\xi,\cJ v)=-\cJ T(\xi,v)$,
\item $T(X,Y)=g(X,\cJ Y)\xi=d\eta(X,Y)\xi$,
\item $\twR(u,v)\xi=0$,
\item $\twR(X,Y)Z=P\Rm(X,Y)Z+\frac{1}{4}g(\cJ Y+\cJ hY,Z)(\cJ X+\cJ hX)-\frac{1}{4}g(\cJ X+\cJ hX, Z)(\cJ Y+\cJ hY)+\frac{1}{2}g(X,\cJ Y)JZ$,
\item $\twR(X,\xi)Z=P\Rm(X,\xi)Z+\frac{1}{2}P(\nabla_X\cJ)Z$.
\end{enumerate}
\end{prop}


\begin{prop}
The followings hold for all $Y$ in $\ker\eta$.
\begin{enumerate}
\item $\overline{\ric}(\xi, Y)=\ric(\xi, Y) =-\frac{1}{2}\left<\di h, JY\right>$,
\item $\overline{\ric}(Y, Y)=\ric(Y,Y)-\left<\Rm(\xi,Y)Y, \xi\right>-\frac{1}{4}|hY|^2+\frac{3}{4}|Y|^2$,
\item $\overline\ric(Y, Z)=\sum_i\left<\Rm(X_i,Y)Z, X_i\right>+\frac{3}{4}\left<Y, Z\right>-\frac{1}{4}\left<hY, hZ\right>$.
\end{enumerate}
\end{prop}









A contact manifold is CR if $J[X,Y]-J[JX,JY]-[JX,Y]-[X,JY]=0$ for all horizontal vector fields $X$ and $Y$. A computation using the first and the second propositions gives

\begin{prop}
A contact metric manifold is CR if and only if
\[
\nabla_u \cJ(v)=\frac{1}{2}\left<u+hu,v\right>\xi-\frac{1}{2}\left<\xi,v\right>(u+hu).
\]
\end{prop}

\end{document}